\documentclass[runningheads]{llncs}
\usepackage[T1]{fontenc}
\usepackage[english]{babel}
\usepackage[utf8]{inputenc}
\usepackage{mathtools}
\usepackage{amssymb}

\usepackage{amsthm}
\usepackage{authblk}
\usepackage[misc]{ifsym}
\usepackage{hyperref}
\usepackage{graphicx}
\usepackage{subcaption}

\usepackage{dirtytalk}
\usepackage{xcolor}
\usepackage{framed}
\definecolor{shadecolor}{RGB}{153,204,255}

\usepackage{enumerate}

\makeatletter
\renewcommand{\definition}{\@spthm{definition}{\definitionname}{\bfseries}{\normalfont}}
\makeatother

\newcommand{\defhigh}[1]{\emph{#1}}
\newcommand{\cupdisjoint}{\overset{\cdot}{\cup}}

\newcommand{\redsolovay}[1][\le]{\ensuremath{{#1}_{\mathrm{S}} } }

\newcommand{\redsolovayzweia}[1][\le]{\ensuremath{{#1}_{\mathrm{S}}^{\mathrm{2a}} } }

\newcommand{\redclopen}[1][\le]{\ensuremath{{#1}_{\mathrm{cL}}^{\mathrm{open}} } }

\makeatletter
\let\c@lemma\c@theorem
\let\c@proposition\c@theorem
\let\c@corollary\c@theorem
\let\c@definition\c@theorem
\let\c@remark\c@theorem
\let\c@conjecture\c@theorem
\@ifundefined{thelemma}{}{}
\@ifundefined{theproposition}{}{}
\@ifundefined{thecorollary}{}{}
\@ifundefined{thedefinition}{}{}
\@ifundefined{theremark}{}{}
\@ifundefined{theconjecture}{}{}
\makeatother

\begin{document}

\title{S2a-reducibility and differentiation in Martin-Löf random reals\thanks{The second author is supported by DFG project 556436876.}}

\author{Georgii Sirotenko\inst{1} \and Ivan Titov\inst{1,2}}
\authorrunning{G.\ Sirotenko and I.\ Titov}
\institute{Universität Heidelberg, Institut für Informatik, 69120 Heidelberg, Germany \and Université de Bordeaux, CNRS, Bordeaux INP, LaBRI, UMR 5800, F-33400 Talence, France}

\maketitle

\begin{abstract}
Solovay reducibility is studied intensively as a tool to compare the approximability and the degree of randomness of left-c.e.\ reals.
By definition, a real is left-c.e.\ if it has a left-c.e.\ approximation, that is, it is the limit of an effective nondecreasing sequence of rationals.
If reals~$\alpha$ and~$\beta$ have left-c.e.\ approximations $a_0,a_1,\dots$ and $b_0,b_1,\dots$, respectively, such that the approximation ratios
\[\frac{\alpha - a_n}{\beta - b_n}\]
are bounded from above by a constant, the real $\alpha$ is Solovay reducible to $\beta$. The latter is the case for any such $\alpha$ and $\beta$ and their left-c.e.\ approximations whenever $\beta$ is Martin-Löf random by the Ku\v{c}era--Slaman Theorem~\cite{Kucera-Slaman-2001}. This result was substantially strengthened by Barmpalias and Lewis-Pye~\cite{Barmpalias-Lewispye-2017}, who demonstrated that, under the given assumptions, the approximation ratios are not only bounded but actually converge to a limit, which does not depend on the considered left-c.e.\ approximations.

Outside the realm of left-c.e.\ reals, Solovay reducibility is viewed as badly behaved \cite[Section~9.2]{Downey-Hirschfeldt-2010}, and is thus rarely used. Accordingly, there is a quest for a suitable extension of Solovay reducibility to the class of all reals: one that coincides with Solovay reducibility on the left-c.e.\ reals but is better behaved when applied to reals in general. Promising candidates include S2a-reducibility on the set of computably approximable reals by Zheng and Rettinger~\cite{Zheng-Rettinger-2004} and monotone Solovay reducibility by Titov~\cite{Titov-2025-CiE}. For the latter, Titov~\cite{Titov-2024} demonstrated that the theorems of Ku\v{c}era and Slaman and of Barmpalias and Lewis-Pye extend to all reals. He further conjectured~\cite[Conjecture~3.2]{Titov-2024} that similar extensions hold for S2a-reducibility in terms of its functional characterization by Kumabe, Miyabe, and Suzuki~\cite{Kumabe-etal-2023}.

In this work, we refute this conjecture by proving that the analogue of the Barmpalias--Lewis-Pye Limit Theorem does not hold for S2a-reducibility.
\end{abstract}

\section{Introduction and background}\label{section:background}

The main objects of interest of computable analysis are \emph{computable real numbers} and \emph{computable real functions}, i.e., real numbers and real-valued functions of a real argument that can be computed (in an appropriate way) by a Turing functional. The relation of such numbers and functions and their properties with randomness notions considered in the field of algorithmic randomness has been intensively studied, in particular for the notion of Martin-Löf randomness~\cite{Brattke-etal-2011}.

\subsection{Representation of reals and computable approximations}\label{subsection:representations}

In this and the subsequent section, we give a short introduction to representations of reals and the formalization of computable functions.

Since one of the most natural ways to represent real numbers is as limit points of sequences of rational numbers (called \emph{approximations}), a basic object of interest in computability theory is computably approximable (also called \emph{effective}) reals and subclasses of them (such as left-c.e.\ reals), which possess a computable approximation with some special properties.
The classes of effective reals which are of interest to us within this topic are defined next.
\begin{definition}\label{def:approximations}
    A \defhigh{computable approximation} is a computable Cauchy sequence of rationals $a_0,a_1,\dots$.
    A \defhigh{left-c.e.}\ and a \defhigh{right-c.e.\ approximation} is a strictly increasing and a strictly decreasing computable approximation, respectively. A \defhigh{d.c.e.\ approximation} is an approximation of the form $a_0-b_0,a_1-b_1,\dots$ where $a_0,a_1,\dots$ and $b_0,b_1,\dots$ are left-c.e.\ approximations.
    A \defhigh{Cauchy name} is a sequence of rationals $a_0,a_1,\dots$ such that $|a_m - a_n|\leq 2^{-n}$ for all $m\geq n$.

    \defhigh{Computably approximable} (or \defhigh{c.a.}) reals are limit points of computable approximations. \defhigh{Left-c.e.}, \defhigh{right-c.e.}, and \defhigh{d.c.e.} reals are limit points of left-c.e., right-c.e., and d.c.e.\ approximations, respectively. \defhigh{Computable} reals are limit points of computable Cauchy names.
\end{definition}

\subsection{Computable functions on rationals and on reals}\label{subsection:computable-functions}

In theoretical computer science, computability notions differ for functions $\mathbb{Q}\to\mathbb{Q}$ and $\mathbb{R}\to\mathbb{R}$ since real numbers in general cannot be encoded finitely. Within the scope of this paper, by computability of a real function we mean the existence of a Turing functional that, given any Cauchy name of $x$, returns a Cauchy name of $f(x)$ (Type-2 computability by Weihrauch~\cite{Weihrauch-2000}).
\begin{definition}\label{def:computable-real-function}
    A function $f:\mathbb{R}\to\mathbb{R}$ is \defhigh{computable (as a real function)} on a set $A$ if there exists a Turing functional with one oracle tape $M$ such that, for every Cauchy name $(q_0,q_1,\dots)$ of a real $x\in A$ given as oracle, $M$ returns a Cauchy name of $f(x)$. In the latter case, we write $f^{(q_0,q_1,\dots)}_n$ for the $n$th element (if defined) on the output tape of $M$ with the oracle $(q_0,q_1,\dots)$ and call it the \defhigh{evaluation of $f(x)$ with precision $n$}.
\end{definition}

\subsection{Differentiability and algorithmic randomness}\label{subsection:differentiability}

It has been known since the early 20th century that functions of bounded variation on $\mathbb{R}$ (Lebesgue, 1904) and Lipschitz continuous functions on $\mathbb{R}^n$ (Rademacher, 1919) are differentiable almost everywhere.

A first effective version of these analytic results was found by Demuth~\cite{Demuth-1975} in 1975. In modern terminology, Demuth showed that every computable function of bounded variation is differentiable at all Martin-Löf random points. We recall the relevant notions.
\begin{definition}\label{def:bounded-variation}
    A function $f$ is of \defhigh{bounded variation} if there exists a constant $C$ such that, for every finite ordered subset $(x_0<x_1<\dots<x_n)$ of $\mathrm{dom}(f)$,
    \[\sum_{i=1}^n|f(x_i) - f(x_{i-1})|\leq C.\]
\end{definition}

Two particular cases of functions of bounded variation are nondecreasing functions and Lipschitz continuous functions defined on a compact interval. Recall that a function $f$ is \defhigh{Lipschitz continuous} if it satisfies $|f(x) - f(y)| \leq L|x-y|$ for some constant $L$ and all $x,y\in\mathrm{dom}(f)$.
\begin{proposition}\label{bounded-variation-is-implied-automatically}
Let $I$ be a compact interval.
\begin{enumerate}[(i)]
    \item Every Lipschitz continuous function $f$ defined on $I$ is of bounded variation.
    \item Every nondecreasing function defined on $I$ is of bounded variation.
\end{enumerate}
\end{proposition}

We use the measure-theoretic characterization of Martin-Löf randomness via uniformly effectively enumerable randomness tests proposed by Martin-Löf~\cite{Martin-Loef-1966} in 1966.
\begin{definition}\label{def:martin-lof-test}
    A \defhigh{Martin-Löf test} is a uniformly computable sequence of open sets, or \defhigh{layers}, $(L_n)_{n\in\mathbb{N}}$ such that for the Lebesgue measure $\mu$, $\mu(L_n)\leq 2^{-n}$ for all $n\in\mathbb{N}$.
    A real $\alpha$ \defhigh{fails} a Martin-Löf test $(L_n)_{n\in\mathbb{N}}$ if $\alpha\in\bigcap_{n\in\mathbb{N}} L_n$.
    A real is \defhigh{Martin-Löf nonrandom} if it fails some Martin-Löf test, and \defhigh{Martin-Löf random} otherwise.
\end{definition}

Martin-Löf~\cite{Martin-Loef-1966} also proved the existence of a Martin-Löf test, called \defhigh{universal}, which is failed by every Martin-Löf nonrandom real. So, to prove the Martin-Löf randomness of a specific real, it suffices to check that it does not fail the universal Martin-Löf test.
\begin{theorem}[Martin-Löf]\label{thm:universal-test}
There exists a Martin-Löf test (called universal) failed by all Martin-Löf nonrandom reals.
\end{theorem}

Recall that a function defined on the reals is \defhigh{differentiable} at the point $x$ if the limit value of the fraction $\frac{f(x) - f(y)}{x-y}$ as $y\to x$ (i) exists and (ii) is finite.

\begin{theorem}[Demuth]\label{differentiability-results:Demuth}
Every computable function of bounded variation $f$ on the reals is differentiable at every Martin-Löf random point.
\end{theorem}

The result of Demuth was extended by Titov~\cite{Titov2025RelativeRA} in 2025 to a wider class of functions. Titov proved that, if a real $\beta$ is Martin-Löf random, then any function of finite variation that is computable at every argument $x<\beta$ (but not necessarily \emph{at} $\beta$) has a left derivative at $\beta$. This result is remarkable from the viewpoint of relative randomness, since it can be applied to some classes of translation functions in terms of Solovay reducibility, which will be explained in the next section.

\section{Solovay reducibility and its versions}\label{section:Solovay-reducibility}

Solovay reducibility was introduced by Solovay~\cite{Solovay-1975} in 1975 as a measure of relative Martin-Löf randomness. Intuitively, a real $\alpha$ is Solovay reducible to another real $\beta$ if there exists a Turing machine that, given any rational input less than $\beta$, computes a rational value less than $\alpha$ that is no farther (up to a multiplicative constant) from $\alpha$ than the input is from $\beta$.
\begin{definition}\label{define:Solovay-reducibility}
    A \defhigh{translation function} from a real $\beta$ to a real $\alpha$ is a computable function $f:\subseteq\mathbb{Q}\to\mathbb{Q}$ that is defined on the left cut of $\beta$ and fulfills $\lim\limits_{q\nearrow\beta} f(q) = \alpha$.

    A real $\alpha$ is \defhigh{Solovay reducible} to a real $\beta$, written~$\alpha\redsolovay\beta$, if there exist a constant $c$ and a translation function $f$ from $\beta$ to $\alpha$ that satisfies the inequality
    \[\alpha - f(q) < c(\beta - q)\quad\text{for all }q<\beta.\]
\end{definition}

\subsection{Solovay reducibility on left-c.e.\ reals}\label{subsection:Solovay-on-left-ce}

On the set of left-c.e.\ reals, Solovay reducibility was characterized by Calude, Hertling, Khoussainov, and Wang~\cite{Calude-etal-1998} in 1998 as a measure of convergence speed of left-c.e.\ approximations.
\begin{proposition}[Calude et al.]\label{Solovay-on-LC-via-approximations}
    A left-c.e.\ real $\alpha$ is Solovay reducible to a left-c.e.\ real $\beta$ iff there exist two left-c.e.\ approximations $a_0,a_1,\dots\nearrow\alpha$ and $b_0,b_1,\dots\nearrow\beta$ and a constant $c$ such that
    \begin{equation}\label{eq:Solovay-on-LC-via-approximations}
        \alpha - a_n < c(\beta - b_n)\quad\text{for all }n.
    \end{equation}
\end{proposition}

In 2001, Ku\v{c}era and Slaman~\cite{Kucera-Slaman-2001} showed that, on the set of left-c.e.\ reals, Martin-Löf random reals form the highest Solovay degree. In 2017, Barmpalias and Lewis-Pye~\cite{Barmpalias-Lewispye-2017} strengthened the Ku\v{c}era--Slaman theorem by proving that, for any two left-c.e.\ approximations $a_0,a_1,\dots$ and $b_0,b_1,\dots$, the ratio $\frac{\alpha - a_n}{\beta - b_n}$ is not only bounded but also converges to a real number called by Miller~\cite{Miller-2017} the \textit{derivative of} $\alpha$ \textit{relative to} $\beta$, which does not depend on the choice of the two left-c.e.\ approximations.
\begin{theorem}[Barmpalias, Lewis-Pye]\label{BLP-original}
    Let $\alpha$ be a left-c.e.\ real and $\beta$ be a Martin-Löf random left-c.e.\ real. Then there exists a real $d$ such that, for any two left-c.e.\ approximations $a_0,a_1,\dots\nearrow\alpha$ and $b_0,b_1,\dots\nearrow\beta$,
    \begin{equation}\label{eq:BLP-original}
        d = \lim_{n\to\infty}\frac{\alpha - a_n}{\beta - b_n}.
    \end{equation}
    Moreover, $\alpha$ is Martin-Löf random iff $d\neq 0$.
\end{theorem}

The proof of the latter result can be split into two logically independent clauses, namely, the \textit{interdiction of unbounded growth} and the \textit{exclusion of infinite oscillation}:
\begin{enumerate}
    \item the unboundedness of the ratio $\frac{\alpha - a_n}{\beta - b_n}$ as $n\to\infty$ would imply the existence of a Martin-Löf test that $\beta$ fails, contradicting its Martin-Löf randomness;
    \item the existence of two constants $c<d$ such that $\frac{\alpha - a_n}{\beta - b_n} < c$ for infinitely many $n$ and $\frac{\alpha - a_n}{\beta - b_n} > d$ for infinitely many $n$ would also imply the existence of a Martin-Löf test that $\beta$ fails, again contradicting its Martin-Löf randomness. Shen and Andreev~\cite{Shen-Andreev-2026} later showed that this clause is an effective corollary of the Bishop's upcrossing inequality~\cite{Bishop-1966}, which is a result from real analysis.
\end{enumerate}

Outside the class of left-c.e.\ reals, the original notion of Solovay reducibility (Definition~\ref{define:Solovay-reducibility}) is not widespread because it does not induce any meaningful degree structure on larger classes of reals.

\subsection{Versions of Solovay reducibility outside of left-c.e.\ reals}\label{subsection:S2a}

A modification of Solovay reducibility extending it to the computably approximable reals was introduced by Zheng and Rettinger~\cite{Zheng-Rettinger-2004} in 2004. Today it is considered by some authors (see e.g.~\cite{Kumabe-etal-2025}) as the standard Solovay reducibility on c.a.\ reals.
\begin{definition}[Zheng, Rettinger, 2004]\label{define:S2a-reducibility}
    A real $\alpha$ is \defhigh{S2a-reducible} to a real $\beta$, written $\alpha\redsolovayzweia\beta$, if there exist two computable approximations $a_0,a_1,\dots$ and $b_0,b_1,\dots$ of $\alpha$ and $\beta$, respectively, and a constant $c$ such that
    \begin{equation}\label{eq:def-S2a}
        |\alpha - a_n| < c(|\beta - b_n| + 2^{-n})\quad\text{for all }n.
    \end{equation}
\end{definition}

Rettinger and Zheng also showed that, on the set of d.c.e.\ reals, Martin-Löf random left-c.e.\ and right-c.e.\ reals form a highest degree; this result was strengthened by Miller~\cite{Miller-2017}: in the same way as in the Barmpalias--Lewis-Pye Limit Theorem, Miller showed that the ratio $\frac{|\alpha - a_n|}{|\beta - b_n|}$ is not only bounded but also convergent.
\begin{theorem}[Miller, 2017]\label{Miller-2017-thm}
    Let $\alpha$ be a d.c.e.\ real and $\beta$ be a Martin-Löf random d.c.e.\ real. Then there exists a real $d$ such that, for any two d.c.e.\ approximations $a_0,a_1,\dots\to\alpha$ and $b_0,b_1,\dots\to\beta$,
    \begin{equation}\label{eq:Miller-2017}
        d = \lim_{n\to\infty}\frac{|\alpha - a_n|}{|\beta - b_n|}.
    \end{equation}
    Moreover, $\alpha$ is Martin-Löf random iff $d\neq 0$.
\end{theorem}

Here, we note that, by~\cite{AMBOSSPIES2000676}, every Martin-Löf random d.c.e.\ real is either left-c.e.\ or right-c.e.; so, the latter theorem cannot be applied to any Martin-Löf random reals other than the left-c.e.\ ones and their additive inverses. By~\cite[Corollary~4.4]{Rettinger-Zheng-2005}, the statement of the latter theorem does not hold in general for c.a.\ reals and arbitrary computable approximations.

A new start was made by Titov~\cite{Titov-2024} in 2024. He demonstrated that the Barmpalias--Lewis-Pye Limit Theorem can be extended to all reals for monotone nondecreasing translation functions. In 2025, a similar result was obtained by the same author~\cite{Titov2025RelativeRA} for translation functions of finite variation defined on rationals or on reals, where the latter are defined in terms of the cl-open-reducibility introduced by Kumabe, Miyabe, and Suzuki~\cite[Definition~5.1]{Kumabe-etal-2023} in 2025.
\begin{definition}[Kumabe et al.]\label{define:cl-open-reducibility}
    An \defhigh{$\mathbb{R}$-translation function} from a real $\beta$ to a real $\alpha$ is a computable function $f\subseteq\mathbb{R}\to\mathbb{R}$ that is defined on $(-\infty,\beta)$ and fulfills $\lim\limits_{x\nearrow\beta}f(x) = \alpha$.
    A real $\alpha$ is \defhigh{cl-open-reducible} to a real $\beta$, written $\alpha\redclopen\beta$, if there exists a Lipschitz continuous $\mathbb{R}$-translation function from $\beta$ to $\alpha$.
\end{definition}

In the following theorem, we summarize Titov's results~\cite{Titov-2024,Titov2025RelativeRA} concerning $\mathbb{Q}$- and $\mathbb{R}$-translation functions of bounded variation.
\begin{theorem}[Titov]\label{Titov-2024}
    Let $\alpha$ be a real and $\beta$ be a Martin-Löf random real. Then there exists a real $d$ such that, for every $\mathbb{Q}$- or $\mathbb{R}$-translation function of bounded variation $f$ from $\beta$ to $\alpha$ (if any exists),
    \begin{equation}\label{eq:BLP-Titov}
        d = \lim_{x\nearrow\beta}\frac{\alpha - f(x)}{\beta - x}.
    \end{equation}
    In particular, if an $\mathbb{R}$-translation function of bounded variation from $\beta$ to $\alpha$ exists, then $\alpha\redclopen\beta$.
\end{theorem}
\begin{remark}\label{bounded-variation-is-crucial}
    The requirement that $f$ have bounded variation is crucial: by~\cite[Proposition~1.11]{Titov-2024}, for every real $\beta$, there exists a $\mathbb{Q}$-translation function of unbounded variation from $\beta$ to itself that does not fulfill~\eqref{eq:BLP-Titov}.
\end{remark}
It is easy to see that, in case the function $f$ in the latter theorem is a real function defined also \emph{at} the point $\beta$, Theorem~\ref{differentiability-results:Demuth} follows automatically.

The proof scheme of Theorem~\ref{Titov-2024} contains the same two independent steps as Theorem~\ref{BLP-original}: both unbounded growth and infinite oscillation of $\frac{\alpha - f(x)}{\beta - x}$ as $x\nearrow\beta$ would imply the construction of a Martin-Löf test that $\beta$ fails.

According to Kumabe, Miyabe, and Suzuki~\cite[Section~5]{Kumabe-etal-2023}, an equivalent characterization of $\redsolovayzweia$ via a computable -- in terms of rational or real computability -- translation function totally defined on some bounded set is impossible. However, a functional characterization can be formulated in terms of \emph{semicomputable} real functions. The concept of semicomputability is explained next.

\subsection{Semicomputability and translation function intervals}\label{subsection:semicomputability}

\begin{definition}\label{def:semicomputability}
    A function $f:\mathbb{R}\to\mathbb{R}$ is \defhigh{lower semicomputable} on a set $A$ if there exists a Turing functional with one oracle tape $M$ such that, for every Cauchy name $(q_0,q_1,\dots)$ of a real $x\in A$ given as oracle, $M$ returns an increasing sequence converging to $f(x)$, and \defhigh{upper semicomputable} if $-f$ is lower semicomputable.
\end{definition}

The following characterizations of lower and upper semicomputability on open intervals are standard.
\begin{proposition}\label{prop:semicomputable-properties}
    Let $I$ be an open interval and $f$ a real function defined on $I$.
    \begin{enumerate}[(i)]
        \item $f$ is lower semicomputable on $I$ iff the set $\{x\in I : f(x) > q\}$ is c.e.\ uniformly in $q\in\mathbb{Q}$.
        \item $f$ is computable iff $f$ is both lower and upper semicomputable.
        \item If $f$ is computable, then $f$ is continuous.
        \item\label{prop:fourth-item}
        If $f$ is lower (resp.\ upper) semicomputable, then $f$ is lower (resp.\ upper) semicontinuous.
    \end{enumerate}
\end{proposition}

A functional characterization of S2a-reducibility via a \defhigh{function interval} consisting of two Lipschitz continuous functions was found by Kumabe, Miyabe, and Suzuki~\cite{Kumabe-etal-2023} in 2024. In this work, we follow Titov's approach by first defining a translation function interval without any additional measure-preserving requirements, and we examine some properties of translation function intervals.
\begin{definition}\label{define-translation-function-interval}
    A \defhigh{translation function interval} $(f,h)$ from a real $\beta$ to a real $\alpha$ is a pair of functions $f,h:\mathbb{R}\to\mathbb{R}$ such that
    \begin{enumerate}[(i)]
        \item\label{def:first-item}
        $f(x)\leq h(x)$ for all $x\in\mathbb{R}$,
        \item\label{def:second-item}
        $f$ is lower semicomputable,
        \item\label{def:third-item}
        $h$ is upper semicomputable,
        \item\label{def:fourth-item}
        $f(\beta) = h(\beta) = \alpha$.
    \end{enumerate}
\end{definition}
\begin{proposition}\label{proposition:limits}
    In Definition~\ref{define-translation-function-interval}, requirements~\eqref{def:first-item}--\eqref{def:fourth-item} imply altogether that
    \[\lim\limits_{x\to\beta}f(x) = f(\beta) = \lim\limits_{x\to\beta}h(x) = h(\beta) = \alpha.\]
\end{proposition}
\begin{proof}
Since $f$ is lower semicomputable and $h$ is upper semicomputable, $f$ is lower semicontinuous and $h$ is upper semicontinuous by Proposition~\ref{prop:semicomputable-properties}\eqref{prop:fourth-item}, so
\[f(\beta)\leq \liminf_{x\to\beta}f(x)\quad\text{and}\quad\limsup_{x\to\beta}h(x)\leq h(\beta).\]
By requirement~\eqref{def:first-item}, $f(x)\leq h(x)$ for all $x\in\mathbb{R}$, hence
\[\liminf_{x\to\beta}f(x)\leq\liminf_{x\to\beta}h(x)\quad\text{and}\quad\limsup_{x\to\beta}f(x)\leq\limsup_{x\to\beta}h(x).\]
By combining these inequalities with requirement~\eqref{def:fourth-item}, we obtain
\begin{align*}
\alpha = f(\beta)&\leq \liminf_{x\to\beta}f(x)\leq\limsup_{x\to\beta}f(x)\leq\limsup_{x\to\beta}h(x)\leq h(\beta) = \alpha\quad\text{and}\\
\alpha = f(\beta)&\leq \liminf_{x\to\beta}f(x)\leq\liminf_{x\to\beta}h(x)\leq\limsup_{x\to\beta}h(x)\leq h(\beta) = \alpha.
\end{align*}
Thus, $\liminf_{x\to\beta}f(x) = \limsup_{x\to\beta}f(x) = \alpha$, hence we obtain $\lim\limits_{x\to\beta}f(x) = \alpha$, and similarly for $h$.
\end{proof}

\begin{proposition}\label{prop:g-upper-semicomputable}
    If $(f,h)$ is a translation function interval from $\beta$ to $\alpha$, then $g = h - f$ is an upper semicomputable function that satisfies $\lim\limits_{x\to\beta}g(x) = g(\beta) = 0$.
\end{proposition}
\begin{proof}
The function $g(x) = h(x) - f(x)$ is upper semicomputable as the sum of two upper semicomputable functions $-f$ and $h$.
By Proposition~\ref{proposition:limits},
\[\lim_{x\to\beta}g(x) = \lim_{x\to\beta}(h(x) - f(x)) = \lim_{x\to\beta}h(x) - \lim_{x\to\beta}f(x) = \alpha - \alpha = 0,\]
and the equality $g(\beta) = h(\beta) - f(\beta) = \alpha - \alpha = 0$ is immediate from~(iv).
\end{proof}


The next proposition shows that the set of computably approximable reals is closed downwards under the relation \say{there exists a translation function from one real to another}.

\begin{proposition}\label{prop:ca-closed-downward}
    If $\beta$ is a c.a.\ real, and there exists a translation function interval from $\beta$ to $\alpha$, then $\alpha$ is c.a.\ as well.
\end{proposition}
\begin{proof}
The proof relies on the following lemma, whose proof is technical and is moved to the appendix.
\begin{lemma}\label{good-approximation-exists}
    If $(f,h)$ is a translation function interval from a c.a.\ real $\beta$ to another real $\alpha$, then there exists a computable approximation $b'_0,b'_1,\dots\to\beta$ such that $h(b'_i) - f(b'_i) < 2^{-i}$ for all $i$.
\end{lemma}

Let $(f,h)$ be a translation function interval from a c.a.\ real $\beta$ to another real $\alpha$, and let $b'_0,b'_1,\dots$ be a computable approximation of $\beta$ as guaranteed by Lemma~\ref{good-approximation-exists}, where $h(b'_i) - f(b'_i) < 2^{-i}$ for every $i$.

We construct a computable approximation of $\alpha$ by defining
\begin{equation}\label{eq:approximate-f}
    a_n = \tilde f(b'_n),
\end{equation}
where $\tilde f(b'_n)$ is the value of $f(b'_n)$ approximated from below with accuracy $2^{-n}$ (which is possible by simultaneously approximating $f(b'_n)$ from below and $h(b'_n)$ from above, since $h(b'_n) - f(b'_n) < 2^{-n}$). The sequence $a_0,a_1,\dots$ is infinite, computable, and has limit $\alpha$ since $\lim\limits_{x\to\beta}f(x) = \alpha$ by Proposition~\ref{proposition:limits}. Thus, $\alpha$ is computably approximable.
\end{proof}

Kumabe, Miyabe, and Suzuki demonstrated~\cite[Theorem~3.7]{Kumabe-etal-2023} in 2025 that S2a-reducibility can be equivalently characterized via the existence of a translation function interval consisting of two Lipschitz continuous functions.
\begin{theorem}[Kumabe et al.]\label{S2a-functional-criterion}
    A c.a.\ real $\alpha$ is S2a-reducible to a c.a.\ real $\beta$ iff there exists a translation function interval $(f,h)$ from $\beta$ to $\alpha$ such that $f$ and $h$ are Lipschitz continuous.
\end{theorem}

In the context of randomness, by Remark~\ref{bounded-variation-is-crucial}, it makes sense to consider only translation function intervals $(f,h)$ where both $f$ and $h$ have bounded variation. In particular, by Proposition~\ref{bounded-variation-is-implied-automatically}, nondecreasing or Lipschitz continuous functions automatically have bounded variation.

\medskip

In 2025, Titov conjectured that the exclusion of infinite oscillation, which is the second logical clause of Theorem~\ref{Titov-2024}, also holds for function intervals.
\begin{conjecture}[Titov, {\cite[Conjecture~3.2]{Titov-2024}}]\label{Titov-Conjecture}
    Let $\alpha$ be a c.a.\ real and $\beta$ be a Martin-Löf random c.a.\ real that fulfills $\alpha\redsolovayzweia\beta$ via a function interval $(f,h)$. Then there exists a constant $d$ such that $f'(\beta) = h'(\beta) = d$, where $d$ does not depend on the choice of the function interval witnessing the reducibility $\alpha\redsolovayzweia\beta$. Moreover, $d=0$ if and only if $\alpha$ is not Martin-Löf random.
\end{conjecture}

In the next section, we prove that the conjecture is incorrect. In fact, both interdictions of unbounded growth and of infinite oscillation, which are the two clauses of the Barmpalias--Lewis-Pye Limit Theorem, fail for function intervals.

\section{The Barmpalias--Lewis-Pye Limit Theorem does not hold for S2a-reducibility}\label{section:main-result}

In this section, we construct a Martin-Löf random real $\beta$ and two function intervals $(-h,h)$ from $\beta$ to $0$ and $(\tilde f, \tilde h)$ from $\beta$ to $\beta$ such that:
\begin{enumerate}[(i)]
    \item $h$ is Lipschitz continuous with Lipschitz constant $L=1$, which (since $-h$ is also Lipschitz continuous with $L=1$) implies that $0\redsolovayzweia\beta$; the functions $\tilde f$ and $\tilde h$ are nondecreasing;
    \item the ratio $\frac{|0 - h(x)|}{|\beta - x|}$ is bounded but oscillates infinitely as $x\nearrow\beta$ and as $x\searrow\beta$, and the same holds for the ratio $\frac{|0 - (-h(x))|}{|\beta - x|}$;
    \item the ratios $\frac{|\beta - \tilde f(x)|}{|\beta - x|}$ and $\frac{|\beta - \tilde h(x)|}{|\beta - x|}$ are unbounded as $x\to\beta$.
\end{enumerate}
In particular, by Proposition~\ref{bounded-variation-is-implied-automatically}, the functions $\tilde f$ and $\tilde h$ have bounded variation, and the function interval $(-h,h)$ refutes Conjecture~\ref{Titov-Conjecture} because $\beta$ is Martin-Löf random and $0\redsolovayzweia\beta$, but the ratio $\frac{|0 - h(x)|}{|\beta - x|}$ exhibits infinite oscillation as $x\to\beta$ on both sides.

To simplify the presentation, for a function $g$ and a real $y$, we denote by $D^g_{y}(x)$ the ratio $\frac{g(y) - g(x)}{y - x}$ for a real $x\neq y$. In particular, $|D^g_{y}(x)| = \frac{|g(y) - g(x)|}{|y - x|}$.

\begin{theorem}\label{BLP-negation}
    There exist a c.a.\ Martin-Löf random real $\beta$, a function interval $(-h,h)$ from $\beta$ to $0$, and a function interval $(\tilde f,\tilde h)$ from $\beta$ to $\beta$ such that
    \begin{align*}
        &h\text{ is Lipschitz continuous, }\tilde f\text{ and }\tilde h\text{ are nondecreasing},\\
        &\liminf_{x\nearrow\beta}|D^h_{\beta}(x)| = \liminf_{x\searrow\beta}|D^h_{\beta}(x)| = 0 < 1 = \limsup_{x\nearrow\beta}|D^h_{\beta}(x)| = \limsup_{x\searrow\beta}|D^h_{\beta}(x)|,\\
        &\limsup_{x\nearrow\beta}|D^{\tilde f}_{\beta}(x)| = \infty = \limsup_{x\searrow\beta}|D^{\tilde h}_{\beta}(x)|.
    \end{align*}
\end{theorem}
\begin{proof}
We prove the theorem by constructing a computable approximation that converges to a Martin-Löf random real $\beta$ and has no elements on intervals arbitrarily close to $\beta$ relative to the interval length. We formalize this in the following lemma.
\begin{lemma}\label{good-real-and-approximation-exist}
    There exists a two-sided computable approximation $b_0,b_1,\dots$ such that $\beta = \lim_{n\to\infty}b_n$ is Martin-Löf random and there exists a sequence of intervals $[l_0,r_0],[l_1,r_1],\dots$ satisfying
    \begin{align}
        &\lim_{i\to\infty}(r_i - l_i) = 0,\label{eq:lemma-i-shrinking}\\
        &\forall k\in\mathbb{N}\ \exists i\ \big(\beta\in[r_i, r_i + \tfrac{1}{k}(r_i-l_i)]\text{ and }b_n\notin(l_i,r_i)\text{ for all }n\big),\label{eq:lemma-ii-right}\\
        &\forall k\in\mathbb{N}\ \exists i\ \big(\beta\in[l_i - \tfrac{1}{k}(r_i-l_i), l_i]\text{ and }b_n\notin(l_i,r_i)\text{ for all }n\big).\label{eq:lemma-ii-left}
    \end{align}
\end{lemma}

\begin{proof}
The idea of the proof is the step-wise construction of a computable approximation $b_0,b_1,\ldots$ satisfying~\eqref{eq:lemma-i-shrinking}--\eqref{eq:lemma-ii-left} for an appropriate sequence of intervals $[l_0,r_0],[l_1,r_1],\dots$ such that the limit point $\beta = \lim_{n\to\infty}b_n$ has the following property: for the \textit{first layer} $U$ of a universal Martin-Löf test (see Theorem~\ref{thm:universal-test}), which is a c.e.\ open set $U=\bigcup_i U_i$ in $[0,1]$ with Lebesgue measure $\mu(U)\leq\tfrac{1}{2}$, we have $\beta\notin U$. This implies the Martin-Löf randomness of $\beta$ since every Martin-Löf nonrandom real is contained in all layers of a universal Martin-Löf test.

For an interval $A$, we denote by $l(A)$, $r(A)$, and $\overline{A}$ the left endpoint, right endpoint, and closure of $A$, respectively.
\paragraph*{Construction of the computable approximation of $\beta$.}
At step $0$, we set the first element of the approximation $b_0 = 0$ and two open intervals $I_0 = (0,1)$ and $J_0 = (1,2)$ and let $\xi(0) = (0)$ be the \defhigh{index chain} of step $0$.

At the beginning of step $t+1$, where $t\geq 0$, we denote $U^{(t)} = \bigcup_{i=0}^t U_i$ and fix the index chain $\xi(t) = (k_0,\dots,k_m)$ of step $t$. We assume the following list of properties of step $t$ and its \defhigh{index chain} as induction hypothesis:
\begin{align}
&k_0 = 0 < k_1 < \dots < k_m = t,\label{eq:induction:index-chain}\\
&b_t\in \overline{I_{k_m}}\subseteq \overline{I_{k_{m-1}}}\subseteq \dots \subseteq \overline{I_{k_0}}\text{ and }\overline{J_{k_i}}\subseteq \overline{I_{k_{i-1}}}\text{ for all }i\in\{1,\dots,m\},\label{eq:induction:nested-intervals}\\
&|I_{k_i}|\leq \tfrac{1}{i}|J_{k_i}|\leq \tfrac{1}{i+1}|I_{k_{i-1}}|\text{ for every }i\in\{1,\dots,m\},\label{eq:induction:|I|=|J|}\\
&\begin{cases}I_{k_i}\subseteq[r(J_{k_i}), r(J_{k_i}) + \tfrac{1}{i}|J_{k_i}|]&\text{if }i\text{ is odd}\\
I_{k_i}\subseteq[l(J_{k_i}) - \tfrac{1}{i}|J_{k_i}|, l(J_{k_i})]&\text{if }i\text{ is even}\end{cases}\quad\text{for every }i\in\{1,\dots,m\},\label{eq:induction:left-and-right}\\
&\frac{\mu(U^{(t)}\cap I_{k_i})}{\mu(I_{k_i})}\leq \tfrac{1}{2}\text{ for every }i\in\{0,\dots,m\},\label{eq:induction:measure-bound}\\
&\nexists t'<t : \xi(t') = (k_0,\dots,k_{i-1}, k'_i,\dots,t')\text{ where }k'_i>k_i\text{ and }1\leq i\leq m,\label{eq:induction:no-greater}\\
&I_t\cap I_{t'} = \emptyset\text{ and }J_t\cap I_{t'} = \emptyset\text{ for all }t'\in\{0,1,\dots,t-1\}\setminus\xi(t),\label{eq:induction:disjoint-intervals}\\
&b_n\notin J_{k_i}\text{ for every }n\leq t\text{ and }i\in\{0,\dots,m\}.\label{eq:induction:J-is-empty}
\end{align}

Note that, for step $0$, all properties \eqref{eq:induction:index-chain} through \eqref{eq:induction:J-is-empty} hold. 

Let $U_{t+1}$ be a new interval enumerated into $U$. Set
\[U^{(t+1)} = U^{(t)}\cup U_{t+1} = \bigcup_{i=0}^{t+1}U_i.\]

First, we fix the greatest second order index $i$ in the range $0,\dots,m$ such that
\begin{equation}\label{eq:fix-index-i}
    \frac{\mu(U^{(t+1)}\cap I_{k_i})}{\mu(I_{k_i})}\leq \tfrac{1}{2},
\end{equation}
define the index chain of $t+1$ as
\[\xi(t+1) = (k_0,\dots,k_i,t+1),\]
and consider the infinite disjoint partition of $I_{k_i}$ as
\begin{equation}\label{eq:partition-into-disjoint-union}
    I_{k_i} = I^{(1)}_{k_i}\cupdisjoint I^{(2)}_{k_i}\cupdisjoint\dots
\end{equation}
where
\[I^{(s)}_{k_i} = \begin{cases}\big[b_{k_i} + (1 - \tfrac{1}{i+2})^s|I_{k_i}|, b_{k_i} + (1 - \tfrac{1}{i+2})^{s-1}|I_{k_i}|\big]&\text{for all even }i,\\
\big[b_{k_i} - (1 - \tfrac{1}{i+2})^{s-1}|I_{k_i}|, b_{k_i} - (1 - \tfrac{1}{i+2})^s|I_{k_i}|\big]&\text{for all odd }i,\end{cases}\]
for all $s\geq 1$.
Since $\mu(I_{k_i}\cap U^{(t+1)})\leq \tfrac{1}{2}\mu(I_{k_i})$ by the choice of $i$, we can fix the minimal index $s\geq 1$ such that
\begin{equation}\label{eq:fix-index-s}
    \mu(I^{(s)}_{k_i}\cap U^{(t+1)})\leq \tfrac{1}{2}\mu(I^{(s)}_{k_i}),
\end{equation}
and let $b_{t+1}$ be the most distant point in $\overline{I^{(s)}_{k_i}}$ from $b_{k_i}$ that is not covered by $U^{(t+1)}$ (so, for instance, $b_{t+1} = r(I^{(s)}_{k_i})$ in case $i$ is even and $r(I^{(s)}_{k_i})\notin U^{(t+1)}$, and $b_{t+1} = l(I^{(s)}_{k_i})$ in case $i$ is odd and $l(I^{(s)}_{k_i})\notin U^{(t+1)}$). Such a point exists by the choice of $i$. Finally, we set the (open) intervals
\[
\begin{cases}I_{t+1} = (l(I^{(s)}_{k_i}), b_{t+1}),\ J_{t+1} = (b_{t}, l(I_{t+1}))&\text{if }i\text{ is even},\\
I_{t+1} = (b_{t+1}, r(I^{(s)}_{k_i})),\ J_{t+1} = (r(I_{t+1}), b_{t})&\text{if }i\text{ is odd}.\end{cases}
\]
\paragraph*{Verification (sketch).}
The sequence $(b_n)_{n\in\mathbb{N}}$ converges since, for every level $i$, all $b_t$ with sufficiently large $t$ have an index chain starting from the same $i+1$ indices $(k_0,\ldots,k_i)$ (we call indices with this property \textit{stable}), and thus lie in the same interval $I_{k_i}\subseteq\cdots\subseteq I_{k_0}$ of length at most $2^{-i}$. Its limit $\beta$ does not lie in $U$ since, for every enumeration step $k$, all $b_t$ with $t>k$ (and thus also $\beta$, by a compactness argument) do not lie in $U^{(k)}$. Hence $\beta$ is Martin-Löf random.

Setting $[l_i,r_i]=\overline{J_i}$, we obtain~\eqref{eq:lemma-ii-right} and~\eqref{eq:lemma-ii-left} witnessed by the stable indices $(k_2,k_4,k_6,\ldots)$ and $(k_1,k_3,k_5,\ldots)$, respectively, since for every $i$ all $b_t$ with $t>i$ (and thus also $\beta$) lie in $\overline{I_{k_i}}$, while no $b_t$ lies in $J_{k_i}$, and $|I_{k_i}|<\tfrac{1}{i}|J_{k_i}|$.
The detailed verification is moved to the appendix.
\end{proof}

\smallskip
Let $\beta$ be a Martin-Löf random real with computable approximation $b_0,b_1,\dots$ and sequence of intervals $[l_0,r_0],[l_1,r_1],\dots$ as in Lemma~\ref{good-real-and-approximation-exist}.
We construct the function interval $(-h,h)$ witnessing $0\redsolovayzweia\beta$ by setting
\begin{equation}\label{eq:def-h}
    h(x) = \inf\{|x - b_i| : i\in\mathbb{N}\}.
\end{equation}

The function $h$ is Lipschitz continuous with constant $1$. It is upper semicomputable via a Turing functional that, given a Cauchy name $(q_0,q_1,\dots)$, returns the nonincreasing sequence $(p_0,p_1,\dots)$ with
\[p_n = \min\{|q_n - b_m| : 0\leq m\leq n\} + 2^{-n}.\]
By construction, $h(b_i) = 0$ for every $i$, and $h(\beta) = 0$ since $\lim_{i\to\infty}|\beta - b_i| = 0$. Since $b_0,b_1,\dots$ is a two-sided approximation of $\beta$, there exists a left-sided subsequence $b_{i_0},b_{i_1},\dots$, giving
\[\liminf_{x\nearrow\beta}\frac{|0 - h(x)|}{|\beta - x|}\leq \lim_{n\to\infty}\frac{|0 - h(b_{i_n})|}{|\beta - b_{i_n}|} = 0.\]

On the other hand, if $x_i = \frac{l_i + r_i}{2}$ for some $i$ satisfying~\eqref{eq:lemma-ii-right}, then by~\eqref{eq:def-h} we have $h(x_i)\geq \frac{r_i - l_i}{2}$ and $|\beta - x_i|\leq \frac{r_i - l_i}{2} + \frac{r_i - l_i}{k}$, where $k$ witnesses Lemma~\ref{good-real-and-approximation-exist} for this $i$. Therefore
\[\frac{|0 - h(x_i)|}{|\beta - x_i|}\geq \frac{(r_i - l_i)/2}{(r_i - l_i)/2 + (r_i - l_i)/k}.\]
Fix indices $j_0,j_1,\dots$ such that, for each $n\geq 1$, $j_n$ satisfies~\eqref{eq:lemma-ii-right} with $k = n$, and set $x_n = \frac{l_{j_n} + r_{j_n}}{2}$. Then $x_n < \beta$ and $x_n\nearrow\beta$, hence

\[\limsup_{x\nearrow\beta}\frac{|0 - h(x)|}{|\beta - x|}\geq \limsup_{n\to\infty}\frac{|0-h(x_n)|}{|\beta - x_n|} \geq \limsup_{n\to\infty}\frac{\frac{r_{j_n} - l_{j_n}}{2}}{\frac{r_{j_n} - l_{j_n}}{2} + \frac{r_{j_n} - l_{j_n}}{n}} = 1.\]


On the other hand, the Lipschitz continuity of $h$ with Lipschitz constant $1$ implies that $\frac{|h(\beta) - h(x)|}{|\beta - x|}\leq 1$ for every $x$; therefore $\limsup_{x\nearrow\beta}\frac{|0 - h(x)|}{|\beta - x|} = 1$.

The equalities $\liminf_{x\searrow\beta}\frac{|0 - h(x)|}{|\beta - x|} = 0$ and $\limsup_{x\searrow\beta}\frac{|0 - h(x)|}{|\beta - x|} = 1$ are obtained analogously from~\eqref{eq:lemma-ii-left}.

The function $-h$ is Lipschitz continuous, lower semicomputable (since $h$ is upper semicomputable), and satisfies $-h(\beta) = 0$; thus we have $0\redsolovayzweia\beta$ via the function interval $(-h,h)$. Note that $-h$ has the same limit superior and limit inferior for the ratio $\frac{|0 - (-h(x))|}{|\beta - x|}$ for $x\nearrow\beta$ and $x\searrow\beta$ as $h$, since $|0 - (-h(x))| = |h(x)| = |0 - h(x)|$.

\smallskip
We now construct the function interval $(\tilde f,\tilde h)$. Let $b_0,b_1,\dots$ be the two-sided computable approximation of the Martin-Löf random real $\beta$ as in Lemma~\ref{good-real-and-approximation-exist}.

The functions $\tilde f$ and $\tilde h$ are defined via their hypograph $U_{\tilde f} = \{(x,y) : \tilde f(x) > y\}$ and epigraph $U_{\tilde h} = \{(x,y) : \tilde h(x) < y\}$, respectively, where $U_{\tilde f} = \bigcup_{i\in\mathbb{N}}U^{(i)}_{\tilde f}$ and $U_{\tilde h} = \bigcup_{i\in\mathbb{N}}U^{(i)}_{\tilde h}$. We construct the sets $U^{(i)}_{\tilde f}$ and $U^{(i)}_{\tilde h}$ stage by stage as follows:
\begin{align*}
U^{(0)}_{\tilde f} &= \{(x,0) : x\in(0,1)\}, & U^{(i)}_{\tilde f} &= U^{(i-1)}_{\tilde f}\cup\{(x,y) : x>b_i\wedge y < b_i\} & (i>0),\\
U^{(0)}_{\tilde h} &= \{(x,1) : x\in(0,1)\}, & U^{(i)}_{\tilde h} &= U^{(i-1)}_{\tilde h}\cup\{(x,y) : x<b_i\wedge y > b_i\} & (i>0).
\end{align*}
By construction, $\tilde f$ and $\tilde h$ are nondecreasing. The function $\tilde f$ is lower semicomputable because its hypograph $\{(x,y) : y < \tilde f(x)\}$ equals the c.e.\ union $U_{\tilde f}$. The function $\tilde h$ is upper semicomputable because its epigraph $\{(x,y) : y > \tilde h(x)\}$ equals the c.e.\ union $U_{\tilde h}$.

We now show that $\tilde f(\beta) = \tilde h(\beta) = \beta$. Consider the strictly increasing subsequence $b^{(l)}_0,b^{(l)}_1,\dots$ of all $b_i$ less than $\beta$ and the strictly decreasing subsequence $b^{(r)}_0,b^{(r)}_1,\dots$ of all $b_i$ greater than $\beta$. Since $\tilde f(b^{(l)}_i) = b^{(l)}_{i-1}$ and $\tilde f(b^{(r)}_i) = b^{(r)}_{i+1}$,
\begin{align*}
\lim_{n\to\infty}\tilde f(b^{(l)}_n) &= \lim_{n\to\infty}b^{(l)}_{n-1} = \beta,\\
\lim_{n\to\infty}\tilde f(b^{(r)}_n) &= \lim_{n\to\infty}b^{(r)}_{n+1} = \beta.
\end{align*}
By monotonicity of $\tilde f$ and $\lim_{n\to\infty}\tilde f(b^{(l)}_n) = \lim_{n\to\infty}\tilde f(b^{(r)}_n) = \beta$, the latter implies $\tilde f(\beta) = \beta$. Similarly, $\tilde h(\beta) = \beta$.

We now show that
\[\limsup_{x\nearrow\beta}|D^{\tilde f}_{\beta}(x)| = \infty\quad\text{and}\quad\limsup_{x\searrow\beta}|D^{\tilde h}_{\beta}(x)| = \infty.\]
We give the explicit proof for the left equality; the right one is obtained analogously. By Lemma~\ref{good-real-and-approximation-exist}~\eqref{eq:lemma-ii-right}, for any $k$ there exists an interval $(l_i,r_i)$ such that $\beta\in[r_i, r_i + \tfrac{1}{k}(r_i - l_i)]$. Since no point $b_n$ lies in $(l_i,r_i)$, we have $\tilde f(r_i) = \tilde f(l_i)\leq l_i$, where the last inequality follows from $\tilde f(x)\leq x$. Considering the sequence $r_{j_n}$ converging to $\beta$ from the left such that, for each $n\geq 1$, $j_n$ satisfies~\eqref{eq:lemma-ii-right} with $k=n$, we obtain
\begin{align*}
\limsup_{x\nearrow\beta}|D^{\tilde f}_{\beta}(x)| &\geq \limsup_{n\to\infty}|D^{\tilde f}_{\beta}(r_{j_n})| = \limsup_{n\to\infty}\frac{\beta - \tilde f(r_{j_n})}{\beta - r_{j_n}}\\
&\geq \limsup_{n\to\infty}\frac{\beta - l_{j_n}}{\beta - r_{j_n}}\geq \limsup_{n\to\infty}\frac{r_{j_n} - l_{j_n}}{\frac{1}{n}(r_{j_n} - l_{j_n})} = \infty,
\end{align*}
using $\beta - r_{j_n}\leq \frac{1}{n}(r_{j_n} - l_{j_n})$ and $\beta\geq r_{j_n}$.
\end{proof}


\begin{remark}\label{remark:no-speedability}
    The existence of a function interval $(\tilde f, \tilde h)$ from a Martin-Löf random real $\beta$ to itself constructed in the proof of Theorem~\ref{BLP-negation} excludes the introduction of any nontrivial speedability notion in terms of the functional characterization of S2a-reducibility (i.e.\ via a function interval).
\end{remark}

\section{Conclusion and future work}\label{section:future-work}

In this paper, we refuted Conjecture~\ref{Titov-Conjecture}, originally stated by Titov~\cite[Conjecture~3.2]{Titov-2024}: we constructed a c.a.\ Martin-Löf random real~$\beta$ together with a function interval $(-h,h)$ from~$\beta$ to~$0$ consisting of two Lipschitz continuous functions such that the relevant approximation ratio $\frac{|0-h(x)|}{|\beta-x|}$ oscillates between $0$ and~$1$ as $x\to\beta$ on both sides. We also constructed a function interval $(\tilde f,\tilde h)$ from $\beta$ to itself consisting of two nondecreasing functions whose corresponding ratios are unbounded, which (as observed in Remark~\ref{remark:no-speedability}) precludes any nontrivial notion of speedability for the functional characterization of S2a-reducibility.

Conceptually, in contrast to the case of computable functions on rationals or on reals, we cannot construct a randomness test failed by a given real from a semicomputable function that grows unboundedly or oscillates infinitely in a neighbourhood of that real.

Several questions remain open.
First, it is unclear for which pairs of reals $(\alpha,\beta)$ there exists a translation function interval from $\beta$ to $\alpha$.
Second, although the analogue of the Barmpalias--Lewis-Pye Limit Theorem fails for S2a-reducibility, it may still be the case that a weaker form holds: for any c.a.\ reals $\alpha,\beta$ with $\alpha\redsolovayzweia\beta$ and $\beta$ Martin-Löf random, does there always \emph{exist} a translation function interval $(f,h)$ from~$\beta$ to~$\alpha$ such that both $f$ and $h$ are at least one-sided differentiable at~$\beta$?
Third, it remains unclear whether Martin-Löf randomness can be \emph{characterized} via differentiability properties of translation functions or function intervals on the c.a.\ reals.
Fourth, beyond Martin-Löf randomness, Theorem~\ref{Titov-2024} suggests that connections analogous to those between Solovay reducibility and randomness may exist between other versions of Solovay reducibility and stronger notions of randomness; this is a promising direction for further investigation.

\bibliography{lipics-v2021-sample-article}

\newpage

\appendix

\section*{Appendix A.\ Proof of Lemma~\ref{good-approximation-exists}}\label{appendix-A}

\begin{proof}
Let $(f,h)$ be a translation function interval from a c.a.\ real $\beta$ to another real $\alpha$, let $b_0,b_1,\dots$ be a computable approximation of $\beta$, and define $g = h - f$.

By Proposition~\ref{prop:g-upper-semicomputable}, the function $g$ fulfills $\lim\limits_{x\to\beta}g(x) = 0$, hence, for every $\varepsilon>0$, there exists $\delta>0$ such that
\begin{equation}\label{eq:g-bound}
g(x)\downarrow\in[0,\varepsilon)\text{ for every }x\in(\beta - \delta, \beta + \delta).
\end{equation}

We define the index sequence $i_0,i_1,\dots$ and a sequence of rationals $b'_0,b'_1,\dots$ as follows: starting from $i_0 = 0$ and $b'_0 = b_{i_0} = b_0$, at each step $n$ we search for a rational $q$ and an index $i>i_{n-1}$ such that
\[q\in[b_i - 2^{-n}, b_i + 2^{-n}]\text{ and }g(q)\downarrow < 2^{-n},\]
where $g(q)\downarrow < 2^{-n}$ means that, for the Cauchy name $(q,q,q,\dots)$ of $q$, the Turing functional that approximates $g(q)$ from above outputs a decreasing sequence with some element smaller than $2^{-n}$, and we set $i_n = i$ and $b'_n = q$.

Each step terminates because, for all sufficiently large $i$, such a rational $q$ exists by~\eqref{eq:g-bound}. Therefore, the sequence $b'_0,b'_1,\dots$ is infinite and converges to $\beta$ since $\lim_{n\to\infty}b_n = \beta$ and $\lim_{n\to\infty}(b_n - b'_n) = 0$. Finally, for every $n$, $h(b'_n) - f(b'_n) = g(b'_n) < 2^{-n}$, which concludes the proof.\qed
\end{proof}

\section*{Appendix B.\ Verification of the construction in Lemma~\ref{good-real-and-approximation-exist}}\label{appendix-B}

\paragraph*{Properties of construction step $t+1$.}
Properties \eqref{eq:induction:index-chain} through \eqref{eq:induction:no-greater} for step $t+1$ and its index chain $(k_0,\dots,k_i,t+1)$ follow directly from the construction and from properties \eqref{eq:induction:index-chain} through \eqref{eq:induction:no-greater} for step $t$.

To prove property \eqref{eq:induction:disjoint-intervals} for step $t+1$, we fix a step $t'\in\{0,1,\dots,t+1\}\setminus\{k_0,\dots,k_i,t+1\}$.
\begin{itemize}
    \item If its index chain has the form $\xi(t') = (k_0,\dots,k_{i-l},k'_{i-l+1},\dots,t')$ where $k_{i-l}\neq k'_{i-l}$ for some $l>0$, then by~\eqref{eq:induction:no-greater} for step $t+1$, we have $k_{i-l}>k'_{i-l}$. From $k_{i-l}>k'_{i-l}$, we obtain by~\eqref{eq:induction:disjoint-intervals} for step $k_{i-l}$ that $I_{k_{i-l}}\cap I_{k'_{i-l}} = \emptyset$, hence by iterative application of~\eqref{eq:induction:nested-intervals},
    \[I_{t'}\subseteq I_{k'_{i-l}},\ I_{k'_{i-l}}\cap I_{k_{i-l}} = \emptyset,\ I_{t+1}\subseteq I_{k_i}\subseteq I_{k_{i-l}},\text{ and }J_{t+1}\subseteq I_{k_i}\subseteq I_{k_{i-l}},\]
    which implies $I_{t'}\cap I_{t+1} = \emptyset$ and $I_{t'}\cap J_{t+1} = \emptyset$.

    \item If the index chain of $t'$ has the form $\xi(t') = (k_0,\dots,k_i,k'_{i+1},\dots,t')$ where $k'_{i+1}\neq k_{i+1}$, then $k'_{i+1}>k_{i+1}$ would contradict~\eqref{eq:induction:no-greater}, so $k'_{i+1}<k_{i+1}$. At steps $k'_{i+1}$ and $t+1$, we partitioned the same interval $I_{k_i}$ into the same disjoint subintervals as in~\eqref{eq:partition-into-disjoint-union}, since the partition does not depend on the step number.

    Recall $I_{t+1} = I^{(s)}_{k_i}$, and fix $s'$ such that $I_{k'_{i+1}}\subseteq I^{(s')}_{k_i}$.

    If $s'>s$, the choice of $s'$ would imply $\frac{\mu(U^{(k'_{i+1})}\cap I^{(s)}_{k_i})}{\mu(I^{(s)}_{k_i})} > \tfrac{1}{2}$; and thus, by $U^{(k'_{i+1})}\subseteq U^{(t+1)}$, also
    \[\frac{\mu(U^{(t+1)}\cap I^{(s)}_{k_i})}{\mu(I^{(s)}_{k_i})} > \tfrac{1}{2},\]
    contradicting the choice of $s$.

    If $s' = s$, and thus $I_{t+1} = I^{(s)}_{k_i} = I^{(s')}_{k_i}\supseteq I_{k'_{i+1}}$, the choice of $i$ would imply $\frac{\mu(U^{(t+1)}\cap I^{(s)}_{k_i})}{\mu(I^{(s)}_{k_i})} = \frac{\mu(U^{(t+1)}\cap I_{k_{i+1}})}{\mu(I_{k_{i+1}})} > \tfrac{1}{2}$, again contradicting the choice of $s$.

    Therefore $s' < s$. If $i$ is even, by construction
    \[r(I^{(0)}_{k_i})>l(I^{(0)}_{k_i}) = r(I^{(1)}_{k_i})>l(I^{(1)}_{k_i}) = r(I^{(2)}_{k_i})>\dots>b_{k_i};\]
    thus $s'<s$ implies $r(J_{t+1}) = l(I_{t+1})< r(I_{t+1}) = r(I^{(s)}_{k_i})\leq l(I^{(s')}_{k_i}) = l(I_{t'})$, and hence $I_{t+1}\cap I_{t'} = \emptyset$ and $J_{t+1}\cap I_{t'} = \emptyset$. The case $i$ odd is symmetric, with
    \[l(I^{(0)}_{k_i})<r(I^{(0)}_{k_i}) = l(I^{(1)}_{k_i})<r(I^{(1)}_{k_i}) = l(I^{(2)}_{k_i})<\dots<b_{k_i}.\]
\end{itemize}

To prove~\eqref{eq:induction:J-is-empty} for step $t+1$, it suffices to check that $b_{t+1}\notin J_{t+1}$, that $b_{t+1}\notin J_{k_j}$ for all $j\in\{0,\dots,i\}$, and that $b_n\notin J_{t+1}$ for all $n\leq t$.
\begin{itemize}
    \item $b_{t+1}\notin J_{t+1}$ is straightforward since $b_{t+1}\in \overline{I_{t+1}}$ and $J_{t+1}\cap \overline{I_{t+1}} = \emptyset$ by construction.

    \item For every $j\in\{0,\dots,i\}$, $b_{t+1}\in \overline{I_{t+1}}\subseteq \overline{I_{k_j}}$ by~\eqref{eq:induction:nested-intervals} for step $t+1$ and $\overline{I_{k_j}}\cap J_{k_j} = \emptyset$ by construction. Therefore $b_{t+1}\notin J_{k_j}$.


    \item Fix $n\leq t$ and consider the index chain $\xi(n)$ of step $n$. For all $n = k_j$ with $j\in\{0,\dots,i\}$, $b_n = b_{k_j}$ is a boundary point of $I_{k_j}$ by the definition of step $k_j$ and $J_{t+1}\subseteq I_{k_i}\subseteq I_{k_j}$ by~\eqref{eq:induction:nested-intervals}, hence $b_n\notin J_{t+1}$.

    If the index chain of step $n$ has the form $\xi(n) = (k_0,\dots,k_{i-l},k'_{i-l+1},\dots,n)$ with $l\in\{0,\dots,j\}$, $k'_{i-l+1}\neq n$, and $k'_{i-l+1}\neq k_{i-l+1}$ in case $l>0$, then $b_n\in \overline{I_{k'_{i-l+1}}}$ by~\eqref{eq:induction:nested-intervals} for step $n$, while $J_{t+1}\cap I_{k'_{i-l+1}} = \emptyset$ by~\eqref{eq:induction:disjoint-intervals}; therefore, $b_n\notin J_{t+1}$.

    Finally, if the index chain of step $n$ has the form $\xi(n) = (k_0,\dots,k_{i-l},n)$ with $l\in\{0,\dots,j\}$, $n\neq k_{i-l+1}$ in case $l>0$, and $n\neq t+1$ in case $l=0$, then $b_n$ is a boundary point of $I_n$ by the construction of step $n$ and $J_{t+1}\cap I_n = \emptyset$ by~\eqref{eq:induction:disjoint-intervals}. Hence $b_n\notin J_{t+1}$.
\end{itemize}

\paragraph*{The sequence converges.}
First, for every construction step $t$ and its index chain $\xi(t) = (k_0=0, k_1,\dots,k_m=t)$, $\overline{I_0} = [0,1]$ and $|I_0| = 1$, hence by~\eqref{eq:induction:|I|=|J|},
\[|I_{k_m}|\leq \tfrac{1}{2}|I_{k_{m-1}}|\leq \dots \leq \tfrac{1}{2^m}|I_0| = \tfrac{1}{2^m}.\]
Thus, by~\eqref{eq:induction:nested-intervals}, in order to prove the convergence of $(b_0,b_1,\dots)$, it suffices to show that there exist infinitely many \defhigh{stable} indices~$\tilde k_m$, i.e., indices~$\tilde k_m$ such that every construction step~$t>\tilde k_m$ has an index chain of the form $(\tilde k_0,\tilde k_1,\dots,\tilde k_m,\dots,t)$.

Assume, for contradiction, that there are only finitely many stable indices, and let $\tilde k_m$ be the maximal stable index. Then, for every $t>\tilde k_m$, the index chain of $t$ has the form $\xi(t) = (\tilde k_0,\dots,\tilde k_m, k^{(t)}_{m+1},\dots,t)$ with $k^{(t)}_{m+1}$ unstable and $I_{k^{(t)}_{m+1}} = I^{(s)}_{\tilde k_m}$ for an appropriate $s\geq 1$, where $I_{\tilde k_m} = I^{(1)}_{\tilde k_m}\cupdisjoint I^{(2)}_{\tilde k_m}\cupdisjoint\dots$ is the partition defined in~\eqref{eq:partition-into-disjoint-union} (which does not depend on the step number).

Thus we can fix the sequence of \defhigh{mind change indices} $(t_1, t_2,\dots)$ such that $\tilde k_m < t_1 < t_2 < \dots$ and
\begin{align*}
\xi(t_1)&= (\tilde k_0,\dots,\tilde k_m, t_1)\text{ with }I_{t_1} = I^{(s_1)}_{\tilde k_m}\text{ for an appropriate }s_1,\\
\xi(t_2-1)&= (\tilde k_0,\dots,\tilde k_m, t_1,\dots),\\
\xi(t_2)&= (\tilde k_0,\dots,\tilde k_m, t_2)\text{ with }I_{t_2} = I^{(s_2)}_{\tilde k_m}\text{ for an appropriate }s_2,\\
&\vdots
\end{align*}
This sequence is infinite since none of $t_1,t_2,\dots$ is stable. Note that $1\leq s_1<s_2<s_3<\dots$ from the discussion above, and that $t_1 = \tilde k_m + 1$ by construction of step $\tilde k_m + 1$ since all indices $\tilde k_0,\dots,\tilde k_m$ in the index chain of step $\tilde k_m$ are stable.

For every $t\geq \tilde k_m$, both index chains of steps $t$ and $t+1$ start with $(\tilde k_0,\dots,\tilde k_m)$. By construction of step $t+1$, this implies
\begin{equation}\label{eq:bound-on-Ukm}
    \frac{\mu(U^{(t)}\cap I_{\tilde k_m})}{\mu(I_{\tilde k_m})}\leq \tfrac{1}{2}.
\end{equation}
By $s_2 > s_1\geq 1$, the choice of $s_2$ at step $t_2$ gives
\begin{equation}\label{eq:a-little-bit-greater}
    \frac{\mu(U^{(t_2)}\cap I^{(1)}_{\tilde k_m})}{\mu(I^{(1)}_{\tilde k_m})} = \tfrac{1}{2} + \varepsilon\quad\text{for some }\varepsilon>0.
\end{equation}

Fix a natural $N$ such that $\frac{1}{2}(1 - \frac{1}{m+2})^N < \frac{\varepsilon}{m+2}$, and consider step $t_N$. From $1\leq s_1<s_2<\dots<s_N$, we know $s_N\geq N$. We assume $m$ is even (the case $m$ odd is symmetric). By the choice of $s_N$,
\[\frac{\mu(U^{(t_N)}\cap I^{(s)}_{\tilde k_m})}{\mu(I^{(s)}_{\tilde k_m})} > \tfrac{1}{2}\quad\text{for every }s\in\{1,\dots,s_N - 1\}.\]
Combining these inequalities for $s\in\{2,\dots,s_N\}$ together with~\eqref{eq:a-little-bit-greater} for $s=1$, and using
\[I_{\tilde k_m} = \big[b_{\tilde k_m}, b_{\tilde k_m} + (1 - \tfrac{1}{m+2})^{s_N}|I_{\tilde k_m}|\big]\cupdisjoint I^{(s_N)}_{\tilde k_m}\cupdisjoint I^{(s_N - 1)}_{\tilde k_m}\cupdisjoint\dots\cupdisjoint I^{(1)}_{\tilde k_m},\]
we obtain
\begin{align*}
\mu(U^{(t_N)}\cap I_{\tilde k_m})&\geq \sum_{s=1}^{s_N}\mu(U^{(t_N)}\cap I^{(s)}_{\tilde k_m}) = \sum_{s=2}^{s_N}\mu(U^{(t_N)}\cap I^{(s)}_{\tilde k_m}) + \mu(U^{(t_N)}\cap I^{(1)}_{\tilde k_m})\\
&> \tfrac{1}{2}\sum_{s=2}^{s_N}\mu(I^{(s)}_{\tilde k_m}) + \tfrac{1}{2}\mu(I^{(1)}_{\tilde k_m}) + \varepsilon\,\mu(I^{(1)}_{\tilde k_m})\\
&= \tfrac{1}{2}\,\mu\Big(\bigcup_{s=1}^{s_N}I^{(s)}_{\tilde k_m}\Big) + \varepsilon\cdot\tfrac{1}{m+2}|I_{\tilde k_m}|\\
&= \big(\tfrac{1}{2}|I_{\tilde k_m}| - \tfrac{1}{2}(1 - \tfrac{1}{m+2})^{s_N}|I_{\tilde k_m}|\big) + \tfrac{\varepsilon}{m+2}|I_{\tilde k_m}|\\
&= \tfrac{1}{}|I_{\tilde k_m}| + \big(\tfrac{\varepsilon}{m+2} - \tfrac{1}{2}(1 - \tfrac{1}{m+2})^{s_N}\big)|I_{\tilde k_m}| > \tfrac{1}{2}|I_{\tilde k_m}|,
\end{align*}
contradicting~\eqref{eq:bound-on-Ukm} for $t = t_N$. Here the last inequality holds because $s_N\geq N$.

\paragraph*{The limit point is Martin-Löf random.}
Since the constructed sequence $b_0,b_1,\dots$ is a computable approximation, denote its limit by $\beta$.

By the previous discussion, in order to prove that $\beta$ is Martin-Löf random, it suffices to show that $\beta$ is not covered by $U$. We argue by contradiction: if $\beta\in U = \bigcup_{i=0}^{\infty}U_i$, since $U$ is c.e.\ open, we can fix a natural $N$ such that
\begin{equation}\label{eq:beta-in-U}
[\beta - \tfrac{1}{N}, \beta + \tfrac{1}{N}]\subseteq U.
\end{equation}

Let $\tilde k_N$ be the $N$th stable index with corresponding index chain $(\tilde k_0,\dots,\tilde k_N)$. Since all $b_t$ with $t>\tilde k_N$ lie in $I_{\tilde k_N}$ by~\eqref{eq:induction:nested-intervals}, by a compactness argument
\begin{equation}\label{eq:beta-in-Iknn}
\beta\in \overline{I_{\tilde k_N}}.
\end{equation}
Applying~\eqref{eq:induction:|I|=|J|} iteratively yields
\begin{equation}\label{eq:Iknn-small}
|I_{\tilde k_N}| \leq \tfrac{N}{N+1}|I_{\tilde k_{N-1}}|\leq \tfrac{N(N-1)}{(N+1)N}|I_{\tilde k_{N-2}}|\leq \dots \leq \tfrac{N!}{(N+1)!}|I_{k_0}| = \tfrac{1}{N+1} < \tfrac{1}{N}.
\end{equation}
From \eqref{eq:beta-in-U} through \eqref{eq:Iknn-small}, $\overline{I_{\tilde k_N}}\subseteq[\beta - \tfrac{1}{N},\beta + \tfrac{1}{N}]\subseteq U$, so
\begin{equation}\label{eq:contradiction-eq}
\frac{\mu(U\cap I_{\tilde k_N})}{\mu(I_{\tilde k_N})} = 1.
\end{equation}
On the other hand, since $\tilde k_N$ is stable, for every $t>\tilde k_N$, by construction of step~$t$,
$\frac{\mu(U^{(t)}\cap I_{\tilde k_N})}{\mu(I_{\tilde k_N})}\leq \tfrac{1}{2}$.
Since $\mu(U^{(t)})\to \mu(U)$, this implies
$\frac{\mu(U\cap I_{\tilde k_N})}{\mu(I_{\tilde k_N})}\leq \tfrac{1}{2}$, contradicting~\eqref{eq:contradiction-eq}.\qed

\begin{figure}[ht]
\centering
\includegraphics[width=0.85\linewidth]{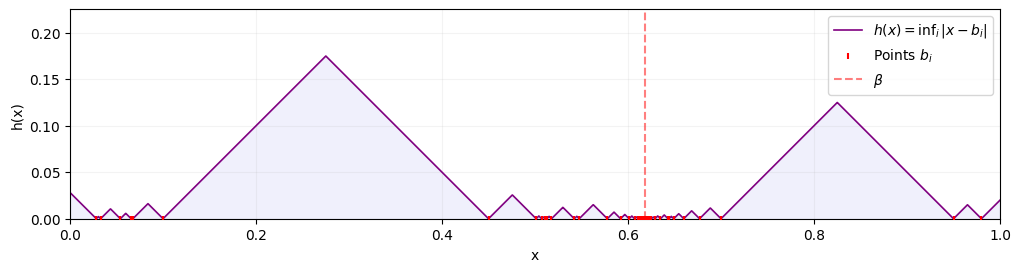}\\[1ex]
\includegraphics[width=0.85\linewidth]{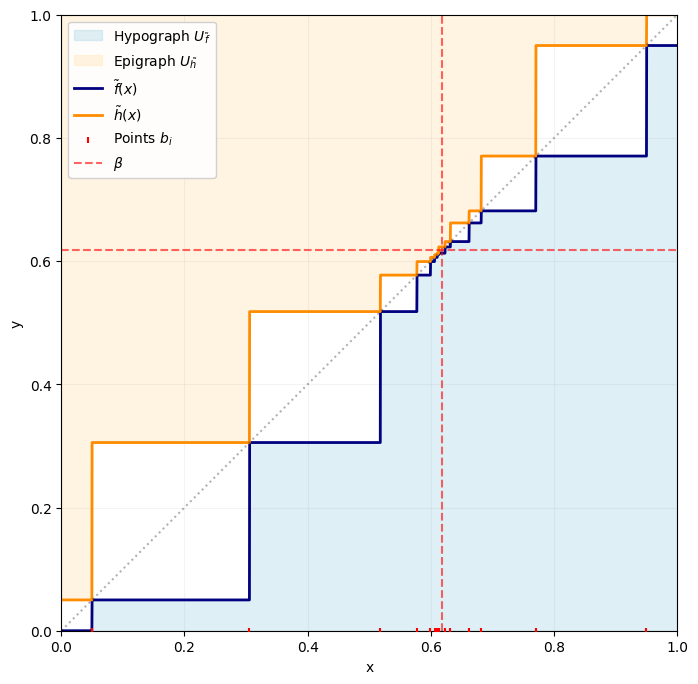}
\caption{\textbf{Upper}: plot of $h$. \textbf{Lower}: plot of $\tilde f$ and $\tilde h$.}
\label{fig:plots}
\end{figure}

\end{document}